\documentclass[11pt, amsfonts, a4paper, oneside]{amsart}

\usepackage{amsmath,amssymb,amsthm}
\usepackage[dvipdfmx,colorlinks=true]{hyperref}
\usepackage[all]{xy}
\usepackage[a4paper]{geometry}
\usepackage{tikz}

\SelectTips{eu}{12}

\newtheorem{theorem}{Theorem}[section]
\newtheorem{proposition}[theorem]{Proposition}
\newtheorem{lemma}[theorem]{Lemma}
\newtheorem{corollary}[theorem]{Corollary}

\theoremstyle{definition}
\newtheorem{definition}[theorem]{Definition}
\newtheorem{example}[theorem]{Example}

\theoremstyle{remark}
\newtheorem{remark}[theorem]{Remark}

\newcommand{\Z}{\mathbb{Z}}
\newcommand{\R}{\mathbb{R}}
\newcommand{\Tor}{\mathrm{Tor}}
\newcommand{\dl}{\mathrm{dl}}
\newcommand{\lk}{\mathrm{lk}}

\title{Two-dimensional Golod complexes}

\author{Kouyemon Iriye}
\address{Department of Mathematical Sciences, Osaka Prefecture University, Sakai, 599-8531, Japan}
\email{kiriye@mi.s.osakafu-u.ac.jp}

\author{Daisuke Kishimoto}
\address{Department of Mathematics, Kyoto University, Kyoto, 606-8502, Japan}
\email{kishi@math.kyoto-u.ac.jp}


\begin{document}

  \begin{abstract}
    We characterize two-dimensional Golod complexes combinatorially by vertex-breakability and topologically by the fat-wedge filtration of a polyhedral product. Applying the characterization, we consider a difference between Golodness over fields and rings, which enables us to give a two-dimensional simple Golod complex over any field such that the corresponding moment-angle complex is not a suspension.
  \end{abstract}

  \baselineskip.525cm

  \maketitle


  \section{Introduction}

  Throughout this paper, let $K$ denote a simplicial complex over the vertex set $[m]=\{1,2,\ldots,m\}$, where ghost vertices are not allowed. Recall that the Stanley-Reisner ring of $K$ over a commutative ring $R$ is defined by
  \[
    R[K]=R[v_1,\ldots,v_m]/(v_{i_1}\cdots v_{i_k}\mid \{i_1,\ldots,i_k\}\not\in K)
  \]
  where we assume that each $v_i$ is of degree 2. It is of particular interest to give a characterization of $K$ which is equivalent to a given algebraic property of $R[K]$. For instance, Cohen-Macaulayness of $R[K]$ is completely characterized by a homological property of $K$. In this paper, we consider a property of $R[K]$, called Golodness, where Golodness was first introduced for a noetherian local ring \cite{G}.

  \begin{definition}
    \label{Golod}
    A simplicial complex $K$ is called \emph{Golod} over $R$ if all products and (higher) Massey products in $\Tor_+^{R[v_1,\ldots,v_m]}(R[K],R)$ vanish, where products and (higher) Massey products are given by the Koszul complex of $R[K]$.
  \end{definition}

  We simply say that $K$ is Golod if it is Golod over any ring.

  Baskakov, Buchstaber and Panov \cite{BBP} showed that there is a space $Z_K$, called the moment-angle complex for $K$, such that
  \begin{equation}
    \label{Z_K}
    H^*(Z_K;R)\cong\Tor_*^{R[v_1,\ldots,v_m]}(R[K],R)
  \end{equation}
  where the isomorphism respects products and (higher) Massey products. This adds a topological viewpoint to the study of Stanley-Reisner rings, which is particularly useful in studying Golodness because $K$ is Golod if $Z_K$ is a suspension. The authors \cite{IK3} developed a nice technology for the study of the homotopy type of $Z_K$, or more generally a polyhedral product, which is called the fat-wedge filtration. In particular, the following is proved in \cite{IK3}, where $\R Z_K$ denotes the real moment-angle complex for $K$.

  \begin{theorem}
    \label{Z_K Tor}
    If the fat-wedge filtration of $\R Z_K$ is trivial, then $Z_K$ is a suspension, implying $K$ is Golod.
  \end{theorem}

  For several important Golod complexes such as the Alexander dual of sequentially Cohen-Macaulay complexes, the fat-wedge filtration of $\R Z_K$ has been proved to be trivial \cite{IK1,IK2,IK3}. In particular, by describing a condition for the fat-wedge filtration of $\R Z_K$ combinatorially, combinatorial characterizations for Golodness of 1-dimensional complexes and triangulations of closed surfaces have been obtained in \cite{IK1,IK3}. Here we recall the result on surface triangulations. We say that $K$ is $k$-neighborly if every $k+1$ vertices of $K$ form a simplex of $K$.

  \begin{theorem}
    \label{surface}
    If $K$ is a triangulation of a connected closed surface, then the following conditions are equivalent:
    \begin{enumerate}
      \item $K$ is Golod;

      \item $K$ is 1-neighborly;

      \item the fat-wedge filtration of $\R Z_K$ is trivial.
    \end{enumerate}
  \end{theorem}

  In this paper, we extend this result to all 2-dimensional simplicial complexes. Clearly, general 2-dimensional simplicial complexes are much more complicated than surface triangulations, and so neighborliness may not be enough to characterize Golodness of 2-dimensional simplicial complexes. Indeed, we have the following example.

  \begin{example}
    Let $K$ be a wedge of two copies of the boundary of a 3-simplex. By definition, $K$ is not 1-neighborly. On the other hand, it follows from \cite[Corollary 7.5]{IK3} that $Z_K$ is of the homotopy type of a wedge of spheres. Then $K$ is Golod.
  \end{example}

  Thus we need to consider a new notion to characterize Golodness of 2-dimensional simplicial complexes. Recall that the full subcomplex of $K$ over a non-empty subset $I\subset[m]$ is defined by $K_I=\{\sigma\in K\mid\sigma\subset I\}$. For a vertex $v\in K$, we write $K_{[m]-v}$ by $\dl_K(v)$. Now we introduce a new notion.

  \begin{definition}
    We say that $K$ has \emph{vertex-breakable} $n$-th homology over an abelian group $A$ if the map
    \[
      \bigoplus_{v\in[m]}(i_v)_*\colon \bigoplus_{v\in[m]}H_n(\dl_K(v);A)\to H_n(K;A)
    \]
    is not surjective, where $i_v\colon\dl_K(v)\to K$ denotes the inclusion. We simply say that $K$ has vertex-breakable $n$-th homology if $K$ has vertex-breakable $n$-th homology over \emph{some} finitely generated abelian group.
  \end{definition}

  \begin{example}
    If $K$ is a triangulation of a connected closed $n$-manifold, then $K$ has vertex-breakable $n$-th homology.
  \end{example}

  Recall that a graph is called chordal if its minimal cycles are of length three. Now we state our results.

  \begin{theorem}
    \label{main}
    If $K$ is a two-dimensional simplicial complex, then the following conditions are equivalent:
    \begin{enumerate}
      \item $K$ is Golod;

      \item the 1-skeleton of $K$ is chordal, and every full subcomplex of $K$ having vertex-breakable second homology is 1-neighborly;

      \item the fat-wedge filtration of $\R Z_K$ is trivial.
    \end{enumerate}
  \end{theorem}

  Golodness was originally defined for local rings by a certain equality involving the Poincar\'e series of their cohomology. Later, Golod \cite{G} proved that the equality is equivalent to vanishing of products and (higher) Massey products as in Definition \ref{Golod}, which enables us to generalize the notion of Golodness over rings. Then it is natural to ask whether or not Golodness over fields and rings are different. Applying Theorem \ref{main}, we will prove that there is certainly a difference between them.

  \begin{theorem}
    \label{M Golod}
    There is a two-dimensional simplicial complex which is Golod over any field but is not Golod over some ring.
  \end{theorem}

  As in Theorems \ref{surface} and \ref{main} as well as \cite{GT,GW,IK1,IK2,IK3}, Golodness over any field of several important classes of simplicial complexes has been proved to be a consequence of the corresponding moment-angle complexes being suspensions. Then it is natural to ask whether or not there is a simplicial complex $K$ such that $K$ is Golod over any field and $Z_K$ is not a suspension. Yano and the first author \cite{IY} proved that there is such a simplicial complex by a direct calculation. We show that a simplicial complex of Theorem \ref{M Golod} is such a simplicial complex too, which is drastically simpler than the one of Yano and the first author.

  \begin{corollary}
    \label{non-suspension}
    There is a two-dimensional simplicial complex $K$ such that $K$ is Golod over any field and $Z_K$ is not a suspension.
  \end{corollary}

  \begin{proof}
    Let $K$ be a simplicial complex of Theorem \ref{M Golod}. The first statement follows from Theorem \ref{M Golod}. If $Z_K$ is a suspension, then by Theorem \ref{Z_K Tor}, $K$ must be Golod over any ring. Thus the second statement also follows from Theorem \ref{M Golod}.
  \end{proof}

  Section 2 recalls properties of the fat-wedge filtration of a polyhedral product that we are going to use, and Section 3 considers a relation between Golodness and vertex-breakability. Section 4 proves Theorem \ref{main}, and Section 5 constructs a triangulation $M$ of the Moore space $S^1\cup_4e^2$ which proves Theorem \ref{M Golod}.\\

  \noindent\textit{Acknoledgement:} The authors were supported by JSPS KAKENHI No. 19K03473 and No. 17K05248.


  \section{Fat-wedge filtration}

  In this section, we recall from \cite{IK3} properties of the fat-wedge filtration of a polyhedral product that we are going to use. First, we define a polyhedral product. Let $(\underline{X},\underline{A})=\{(X_i,A_i)\}_{i=1}^m$ be a collection of pairs of spaces. For a subset $\sigma\subset[m]$, let
  \[
    (\underline{X},\underline{A})^\sigma=Y_1\times\cdots\times Y_m,\quad\text{where}\quad Y_i=\begin{cases}X_i&i\in\sigma\\A_i&i\not\in\sigma.\end{cases}
  \]
  The polyhedral product of $(\underline{X},\underline{A})$ over $K$ is defined by
  \[
    Z_K(\underline{X},\underline{A})=\bigcup_{\sigma\in K}(\underline{X},\underline{A})^\sigma.
  \]
  Clearly, $Z_K(\underline{X},\underline{A})$ is natural with respect to $(\underline{X},\underline{A})$ and inclusions of subcomplexes of $K$. In particular, for $\emptyset\ne I\subset[m]$, $Z_{K_I}(\underline{X}_I,\underline{A}_I)$ is assumed to be a subspace of $Z_K(\underline{X},\underline{A})$, where $(\underline{X}_I,\underline{A}_I)=\{(X_i,A_i)\}_{i\in I}$. For a collection of pointed spaces $\underline{X}=\{X_i\}_{i=1}^m$, let $(C\underline{X},\underline{X})=\{(CX_i,X_i)\}_{i=1}^m$. The polyhedral product $Z_K(C\underline{X},\underline{X})$ is particularly important. Indeed, the moment-angle complex and the real moment-angle complex for $K$ are defined by
  \[
    Z_K=Z_K(D^2,S^1)\quad\text{and}\quad\R Z_K=Z_K(D^1,S^0),
  \]
  which play the fundamental role in toric topology. We refer to a comprehensive survey \cite{BBC} for basic properties of polyhedral products.

  Next we define the fat-wedge filtration of $Z_K(\underline{X},\underline{A})$. Let
  \[
    Z_K^i(\underline{X},\underline{A})=\{(x_1,\ldots,x_m)\in Z_K(\underline{X},\underline{A})\mid\text{at least }m-i\text{ of }x_k\text{ are basepoints}\}
  \]
  for $0\le i\le m$. Clearly,
  \[
    Z_K^i(\underline{X},\underline{A})=\bigcup_{I\subset[m],\,|I|=i}Z_{K_I}(\underline{X}_I,\underline{A}_I).
  \]
  The following is proved in \cite[Theorem 3.1]{IK3}.

  \begin{theorem}
    \label{cone decomposition}
    For each $\emptyset\ne I\subset[m]$, there is a map $\varphi_{K_I}\colon|K|\to\R Z_{K_I}^{|I|-1}$ such that
    \[
      \R Z_K^i=\R Z_K^{i-1}\bigcup_{I\subset[m],\,|I|=i}C|K_I|
    \]
    where the attaching maps are $\varphi_{K_I}$.
  \end{theorem}

  By the construction \cite[Section 5]{IK3} of $\varphi_K$, we have the following naturality.

  \begin{lemma}
    \label{naturality}
    Let $L$ be a subcomplex of $K$ such that the vertex set of $L$ is the same as $K$. Then there is a commutative diagram
    \[
      \xymatrix{
        |L|\ar[r]^{\varphi_L}\ar[d]&\R Z_L^{m-1}\ar[d]\\
        |K|\ar[r]^{\varphi_K}&\R Z_K^{m-1}.
      }
    \]
  \end{lemma}

  We say that the fat-wedge filtration of $\R Z_K$ is trivial if $\varphi_{K_I}$ is null-homotopic for each $\emptyset\ne I\subset[m]$. The main property of the fat-wedge filtration that we use is the following \cite[Theorem 1.2]{IK3}, and Theorem \ref{Z_K Tor} is its immediate corollary.

  \begin{theorem}
    \label{decomposition}
    If the fat-wedge filtration of $\R Z_K$ is trivial, then for any $\underline{X}=\{X_i\}_{i=1}^m$, there is a homotopy equivalence
    \[
      Z_K(C\underline{X},\underline{X})\simeq\bigvee_{\emptyset\ne I\subset[m]}|\Sigma K_I|\wedge\widehat{X}^I
    \]
    where $\widehat{X}^I=\bigwedge_{i\in I}X_i$.
  \end{theorem}

  To show that the map $\varphi_K$ is null-homotopic, the following property is useful, which is proved in the proof of \cite[Theorem 7.2]{IK3}. Recall that a subset $\sigma\subset[m]$ is called a minimal non-face of $K$ if $\sigma\not\in K$ and $\sigma-i\in K$ for all $i\in\sigma$.

  \begin{lemma}
    \label{factorization}
    Let $\widehat{K}$ be a simplicial complex obtained by adding all minimal non-faces to $K$. Then the map $\varphi_K$ factors through the inclusion $|K|\to|\widehat{K}|$.
  \end{lemma}

  We will use the following criterion for triviality of the fat-wedge filtration of $\R Z_K$ \cite[Theorem 1.6]{IK3}.

  \begin{theorem}
    \label{neighborly}
    If $K$ is $\lceil\frac{\dim K}{2}\rceil$-neighborly, then the fat-wedge filtration of $\R Z_K$ is trivial.
  \end{theorem}


  \section{Vertex-breakability}

  In this section, we prove a relation between vertex-breakability and Golodness. To this end, we recall a combinatorial description of the product in $\Tor_*^{R[v_1,\ldots,v_m]}(R[K],R)$. By the classical theorem of Hochster, there is an isomorphism of $R$-modules
  \[
    \Tor^{R[v_1,\ldots,v_m]}_i(R[K],R)\cong\bigoplus_{\emptyset\ne I\subset[m]}\widetilde{H}^{i-|I|-1}(K_I;R).
  \]
  It is remarkable that Baskakov, Buchstaber and Panov \cite{BBP} proved that the product in $\Tor^{R[v_1,\ldots,v_m]}_*(R[K],R)$ is nicely described through this isomorphism as follows. For disjoint simplicial complexes $K,L$, let $K*L$ denote the join of $K$ and $L$, that is,
  \[
    K*L=\{\sigma\sqcup\tau\mid\sigma\in K,\,\tau\in L\}.
  \]
  Then $|K*L|=|K|*|L|\simeq\Sigma|K|\wedge|L|$. For $\emptyset\ne I,J\subset[m]$, we define a map $m_{I,J}\colon K_{I\cup J}\to K_I*K_J$ by
  \[
    m_{I,J}(\sigma)=\sigma_I\sqcup\sigma_J
  \]
  for $I\cap J=\emptyset$ and $m_{I,J}=*$ for $I\cap J\ne\emptyset$.

  \begin{theorem}
    \label{Hochster}
    The product in $\Tor^{R[v_1,\ldots,v_m]}_*(R[K],R)$ is identified with
    \[
      m_{I,J}^*\colon\widetilde{H}^{i-|I|-1}(K_I;R)\otimes\widetilde{H}^{j-|J|-1}(K_J;R)\to\widetilde{H}^{i+j-|I|-|J|-1}(K_{I\cup J};R).
    \]
  \end{theorem}

  We consider a map which is trivial in cohomology.

  \begin{lemma}
    \label{dual trivial}
    Let $f\colon X\to Y$ be a map between CW-complexes of finite type. If a map $f^*\colon H^n(Y;R)\to H^n(X;R)$ is trivial for any ring $R$, then the map $f_*\colon H_n(X;A)\to H_n(Y;A)$ is trivial for any finitely generated abelian group $A$.
  \end{lemma}

  \begin{proof}
    It is sufficient to prove the statement when $A$ is a cyclic group. First, we consider the case $A=\Z/p^r$. As in \cite[p. 239 - 240]{N}, $\Z/p^r$ is injective in the category of $\Z/p^r$-modules. Then
    \[
      H^*(Z;\Z/p^r)\cong\mathrm{Hom}(H_*(Z;\Z/p^r),\Z/p^r)
    \]
    for any space $Z$. This readily implies that $f$ is trivial in homology over $\Z/p^r$.

    Next, we consider the case $A=\Z$. By the universal coefficient theorem, for any abelian group $B$, there is a commutative diagram with exact rows:
    \[
      \xymatrix{
        0\ar[r]&\mathrm{Ext}(H_{n-1}(Y;\Z),B)\ar[r]\ar[d]^{(f_*)^*}&H^n(Y;B)\ar[r]\ar[d]^{f^*}&\mathrm{Hom}(H_n(Y;\Z),B)\ar[r]\ar[d]^{(f_*)^*}&0\\
        0\ar[r]&\mathrm{Ext}(H_{n-1}(X;\Z),B)\ar[r]&H^n(X;B)\ar[r]&\mathrm{Hom}(H_n(X;\Z),B)\ar[r]&0
      }
    \]
    Since the middle $f^*$ is trivial for $B=\Z$ by assumption, the right $(f_*)^*$ is trivial for $B=\Z$. For any finitely generated abelian group $C$, there is a natural isomorphism
    \[
      \mathrm{Hom}(C,\Z)\cong\mathrm{Hom}(C/\mathrm{Tor}(C),\Z)
    \]
    where $\mathrm{Tor}(C)$ denotes the torsion part of $C$. Then since $C/\mathrm{Tor}(C)$ is a free abelian group, the map $f_*\colon H_n(X;\Z)/\mathrm{Tor}(H_n(X;\Z))\to H_n(Y;\Z)/\mathrm{Tor}(H_n(Y;\Z))$ is trivial, implying $f_*(H_n(X;\Z))\subset\mathrm{Tor}(H_n(Y;\Z))$. On the other hand, the right $(f_*)^*$ is trivial for $B=\Z/p^r$, $r$ is arbitrary. Since $H_n(Y;\Z)$ is finitely generated, if we take $r$ large enough, then we can see that the $p$-torsion part $f_*(H_n(X;\Z))$ is trivial. Thus we obtain that $f_*\colon H_n(X;\Z)\to H_n(Y;\Z)$ is trivial.
  \end{proof}

  \begin{lemma}
    \label{vertex-breakable}
    Given a non-trivial finitely generated abelian group $A$, suppose that for two vertices $v,w$ of $K$, the map
    \begin{equation}
      \label{dl}
      (i_v)_*\oplus(i_w)_*\colon H_n(\dl_K(v);A)\oplus H_n(\dl_K(w);A)\to H_n(K;A)
    \end{equation}
    is not surjective. Then $\{v,w\}$ is an edge of $K$ if and only if the map
    \[
      (m_{I,J})_*\colon H_n(K;A)\to H_n(K_I*K_J;A)
    \]
    is trivial for $I=\{v,w\}$ and $J=[m]-\{v,w\}$.
  \end{lemma}

  \begin{proof}
    We set notation. The link of a vertex $u$ of $K$ is defined by
    $$\lk_K(u)=\{\sigma\in K\mid u\not\in\sigma\text{ and }\sigma\cup\{u\}\in K\}.$$
    For an $n$-chain $c=\sum_ia_i[j_{i,0},\ldots,j_{i,n}]$ of $\lk_K(u)$ for $a_i\in A$ and $[j_{i,0},\ldots,j_{i,n}]\in\lk_K(u)$, we abbreviate the $(n+1)$-chain $\sum_ia_i[u,j_{i,0},\ldots,j_{i,n}]$ of $K$ by $u*c$.

    Assume $\{v,w\}$ is not an edge of $K$. Let $c$ be an $n$-cycle of $K$ representing a homology class which is not in the image of the map \eqref{dl}. Then by the assumption above, there are $(n-1)$-chain $c_v$ of $\lk_K(v)$, $(n-1)$-chain $c_w$ of $\lk_K(w)$ and an $n$-chain $d$ of $K_J$ such that
    $$c=v*c_v+w*c_w+d.$$
    Since $c_v$ is a chain of $\lk_K(v)$, $c_v$ is a chain of $K_J$ by the assumption above. One also gets $c_w$ is a chain of $K_J$. Since $\partial c=0$, one has $c_v-v*\partial c_v+c_w-w*\partial c_w+\partial d=0$, implying
    $$\partial c_v=\partial c_w=c_v+c_w+\partial d=0.$$
    Then it follows that
    $$(m_{I,J})_*([c])=[v*c_v-w*c_v+\partial(w*d)]=[v*c_v-w*c_v]\in H_n(K_I*K_J;A).$$
    Since $K_I*K_J=\Sigma K_J$ by assumption, the map
    $$H_{n-1}(K_J;A)\to H_n(K_I*K_J;A),\quad x\mapsto v*x-w*x$$
    is an isomorphism. Thus since $c_v$ is a cycle of $K_J$, if $(m_{I,J})_*([c])=0$, then there is an $n$-chain $e$ of $K_J$ such that $c_v=\partial e$, implying
    $$\partial(w*c_w+d+e)=c_w+\partial d+\partial e=c_w+\partial d+c_v=0,\quad\partial(v*c_v-e)=c_v-\partial e=0.$$
    Therefore $w*c_w+d+e$ and $v*c_v-e$ are cycles of $\dl_K(v)$ and $\dl_K(w)$, respectively, such that $[c]=(i_v)_*([w*c_w+d+e])+(i_w)_*([v*c_v-e])$. This contradicts the definition of $c$, so the only if part is proved. The if part is obvious because $H_n(K_I*K_J;A)=0$.
  \end{proof}

  \begin{proposition}
    \label{1->2}
    If $\dim K\le 2$ and $K$ is Golod, then the following statements hold:
    \begin{enumerate}
      \item the 1-skeleton of $K$ is chordal;

      \item every full subcomplex of $K$ having vertex-breakable second homology is 1-neighborly.
    \end{enumerate}
  \end{proposition}

  \begin{proof}
    The first statement is proved in the proof of \cite[Proposition 8.17]{IK3}. Suppose that $K_I$ has vertex-breakable second homology. Clearly, the second statement holds for $|I|\le 2$, and so we assume $|I|\ge 3$. Take any two vertices $v,w\in I$, and let $J=I-\{v,w\}$. Assume that $\{v,w\}$ is not an edge of $K$. Since $H^*(K_{\{v,w\}};R)$ is a free $R$-module, the strong form of the K\"unneth formula holds as
    \[
      \widetilde{H}^n(K_{\{v,w\}}*K_J;R)\cong\bigoplus_{i+j=n-1}\widetilde{H}^i(K_{\{v,w\}};R)\otimes \widetilde{H}^j(K_J;R).
    \]
    Then it follows from Lemma \ref{dual trivial} that the map $(m_{\{v,w\},J})_*\colon H_2(K_I;A)\to H_2(K_{\{v,w\}}*K_J;A)$ is trivial for any finitely generated abelian group $A$. Thus by Lemma \ref{vertex-breakable}, $\{v,w\}$ is an edge of $K_I$. This is a contradiction, and so $\{v,w\}$ must be an edge of $K$.
  \end{proof}


  \section{Proof of Theorem \ref{main}}

  We will use the following simple lemmas.

  \begin{lemma}
    \label{injective}
    In a commutative diagram of abelian groups
    \[
      \xymatrix{
        0\ar[r]&A_1\ar[r]\ar[d]^{f_1}&A_2\ar[r]\ar[d]^{f_2}&A_3\ar[r]\ar[d]^{f_3}&0\\
        0\ar[r]&B_1\ar[r]&B_2\ar[r]&B_3\ar[r]&0
      }
    \]
    where rows are exact, suppose $f_3$ is injective. Then $f_1$ is injective if and only if so is $f_2$.
  \end{lemma}

  \begin{proof}
    By the snake lemma, there is an exact sequence
    \[
      0\to\mathrm{Ker}\,f_1\to\mathrm{Ker}\,f_2\to\to\mathrm{Ker}\,f_3.
    \]
    Since $f_3$ is injective, $f_1$ is injective if and only if so is $f_2$, completing the proof.
  \end{proof}

  \begin{lemma}
    \label{split}
    For an exact sequence $0\to A_1\to A_2\to A_3\to 0$ of abelian groups, the following statements are equivalent:
    \begin{enumerate}
      \item the map $\mathrm{Ext}(A_3,A)\to\mathrm{Ext}(A_2,A)$ is injective for any abelian group $A$;

      \item the exact sequence $0\to A_1\to A_2\to A_3\to 0$ splits.
    \end{enumerate}
  \end{lemma}

  \begin{proof}
    Consider the exact sequence
    \begin{multline*}
      0\to\mathrm{Hom}(A_3,A)\to\mathrm{Hom}(A_2,A)\to\mathrm{Hom}(A_1,A)\\
      \to\mathrm{Ext}(A_3,A)\to\mathrm{Ext}(A_2,A)\to\mathrm{Ext}(A_1,A)\to 0.
    \end{multline*}
    Then the first statement is equivalent to that the map $\mathrm{Hom}(A_2,A)\to\mathrm{Hom}(A_1,A)$ is surjective for any abelian group $A$. By setting $A=A_1$, this turns out to be equivalent to the second statement.
  \end{proof}

  \begin{lemma}
    \label{Tor surjection}
    Let $f\colon A\to B$ be a surjection between finitely generated abelian groups such that $f_*\colon\mathrm{Tor}(A,\Z/p^r)\to\mathrm{Tor}(B,\Z/p^r)$ is surjective for any prime $p$ and any positive integer $r$. Then $f$ admits a section.
  \end{lemma}

  \begin{proof}
    Since $B$ is finitely generated, there is a decomposition
    \[
      B\cong\mathrm{Free}(B)\oplus\bigoplus_p\mathrm{Tor}_p(B)
    \]
    where $\mathrm{Free}(B)$ is a free abelian group, $\mathrm{Tor}_p(B)$ is the $p$-torsion part of $B$ and $p$ ranges over all primes. Since $f$ is surjective, there is a map $s\colon\mathrm{Free}(B)\to A$ such that $f\circ s=1_{\mathrm{Free}(B)}$. Let $g_i$ be a generator of $\Z/p^{r_i}$ in $\mathrm{Tor}_p(B)\cong\Z/p^{r_1}\oplus\cdots\oplus\Z/p^{r_n}$. Then since
    \begin{equation}
      \label{Tor}
      \mathrm{Tor}(B,\Z/p^r)=\{x\in B\mid p^rx=0\}
    \end{equation}
    and $f_*\colon\mathrm{Tor}(A,\Z/p^{r_i})\to\mathrm{Tor}(B,\Z/p^{r_i})$ is surjective, there is an element $h_i$ of $A$ such that $h_i$ is of order $r_i$ and $f(h_i)=g_i$. Thus there is a map $s_p\colon\mathrm{Tor}_p(B)\to A$ such that $f\circ s_p=1_{\mathrm{Tor}_p(B)}$. Therefore we obtain a section $s\oplus\bigoplus_ps_p\colon B\to A$ of $f$.
  \end{proof}

  \begin{lemma}
    \label{dual}
    Let $f\colon X\to Y$ be a map between finite complexes. If $f_*\colon H_*(X;A)\to H_*(Y;A)$ is surjective for any finitely generated abelian group $A$ and $*=n-1,n$. Then $f^*\colon H^n(Y;A)\to H^n(X;A)$ is injective for any finitely generated abelian group $A$.
  \end{lemma}

  \begin{proof}
    By the universal coefficient theorem, there is a commutative diagram
    \[
      \xymatrix{
        0\ar[r]&\mathrm{Ext}(H_{n-1}(Y;\Z),A)\ar[r]\ar[d]^{(f_*)^*}&H^n(Y;A)\ar[r]\ar[d]^{f^*}&\mathrm{Hom}(H_n(Y;\Z),A)\ar[r]\ar[d]^{(f_*)^*}&0\\
        0\ar[r]&\mathrm{Ext}(H_{n-1}(X;\Z),A)\ar[r]&H^n(X;A)\ar[r]&\mathrm{Hom}(H_n(X;\Z),A)\ar[r]&0
      }
    \]
    where rows are exact. By Lemma \ref{injective}, it is sufficient to show that both left and right $(f_*)^*$ are injective. Since $f_*\colon H_n(X;\Z)\to H_n(Y;\Z)$ is surjective, the right $(f_*)^*$ is injective. By Lemma \ref{split}, the left $(f_*)^*$ is injective if and only if the map $f_*\colon H_{n-1}(X;\Z)\to H_{n-1}(Y;\Z)$ has a section. By the universal coefficient theorem, there is also a commutative diagram
    \[
      \xymatrix{
        0\ar[r]&H_n(X;\Z)\otimes A\ar[r]\ar[d]^{f_*\otimes 1}&H_n(X;A)\ar[r]\ar[d]^{f_*}&\mathrm{Tor}(H_{n-1}(X;\Z),A)\ar[r]\ar[d]^{f_*}&0\\
        0\ar[r]&H_n(X;\Z)\otimes A\ar[r]&H_n(Y;A)\ar[r]&\mathrm{Tor}(H_{n-1}(Y;\Z),A)\ar[r]&0
      }
    \]
    where rows are exact. Since the middle $f_*$ is surjective, so is the right $f_*$. Note that \eqref{Tor} holds for any abelian group $B$. Then since $H_{n-1}(X;\Z)$ and $H_{n-1}(Y;\Z)$ are finitely generated abelian groups, it follows from Lemma \ref{Tor surjection} that the surjectivity of the right $f_*$ implies the existence of a section of the map $f_*\colon H_{n-1}(X;\Z)\to H_{n-1}(Y;\Z)$.
  \end{proof}

  We apply Lemma \ref{dual} to our case.

  \begin{proposition}
    \label{dual special}
    If the 1-skeleton of $K$ is chordal and the map
    \begin{equation}
      \label{homology}
      \bigoplus_{v\in[m]}(i_v)_*\colon \bigoplus_{v\in[m]}H_2(\dl_K(v);A)\to H_2(K;A)
    \end{equation}
    is surjective for any finitely generated abelian group $A$, then the map
    \begin{equation}
      \label{cohomology}
      \bigoplus_{v\in[m]}i_v^*\colon H^2(K;A)\to\bigoplus_{v\in[m]}H^2(\dl_K(v);A)
    \end{equation}
    is injective.
  \end{proposition}

  \begin{proof}
    If $K$ is the boundary of a 2-simplex, then $H_2(K;A)=0$ and $H^2(K;A)=0$ for any abelian group $A$. Thus the statement holds obviously. Assume that $K$ is not the boundary of a 2-simplex. By Lemma \ref{dual}, it is sufficient to show that the map
    \[
      \bigoplus_{v\in[m]}(i_v)_*\colon \bigoplus_{v\in[m]}H_1(\dl_K(v);A)\to H_1(K;A)
    \]
    is surjective for any finitely generated abelian group $A$. For $m\le 3$, $H_1(K;A)=0$ since $K$ is not the boundary of a 2-simplex, where $m$ is the number of vertices of $K$. Then the claim holds. For $m\ge 4$, every minimal cycle of $K$ is in $\dl_K(v)$ for some $v$ since the 1-skeleton of $K$ is chordal. Then the claim also holds.
  \end{proof}

  We consider triviality of the fat-wedge filtration of $\R Z_K$.

  \begin{proposition}
    \label{2->3}
    Suppose $\dim K\le 2$ and $\varphi_{K_I}$ is null-homotopic for all $I\subset[m]$ with $I\ne\emptyset,[m]$. If the 1-skeleton of $K$ is chordal and the map \eqref{homology} is surjective for any finitely generated abelian group $A$, then $\varphi_K$ is null homotopic.
  \end{proposition}

  \begin{proof}
    By definition, $\R Z_K^{m-1}$ is path-connected, and then it has the universal cover $U_K$. Since $K$ is chordal, $|\widehat{K}|$ is simply-connected as in \cite[Proposition 8.17]{IK3}. Then by Lemma \ref{factorization}, $(\varphi_K)_*\colon\pi_1(|K|)\to\pi_1(\R Z_K^{m-1})$ is trivial for any basepoint of $|K|$. In particular, we get a lift $\widetilde{\varphi}_K\colon|K|\to U_K$. Choosing any vertex $v$ of $K$, let $L=\dl_K(v)\sqcup v$. Then by arguing verbatim as above, we get a lift $\widetilde{\varphi}_L\colon|L|\to U_L$ of $\varphi_L\colon|L|\to\R Z_L^{m-1}$. Since there is a commutative diagram
    \[
      \xymatrix{
        |L|\ar[r]^{\varphi_L}\ar[d]&\R Z_L^{m-1}\ar[d]\\
        |K|\ar[r]^{\varphi_K}&\R Z_K^{m-1},
        }
    \]
    it follows from the uniqueness of lifts of $\varphi_L$ and $\varphi_K$ that the square diagram
    \[
      \xymatrix{
        |L|\ar[r]^{\widetilde{\varphi}_L}\ar[d]&U_L\ar[d]\\
        |K|\ar[r]^{\widetilde{\varphi}_K}&U_K
      }
    \]
    commutes. Let $A=\pi_2(\R Z_K^{m-1})$ and $B=\pi_2(\R Z_L^{m-1})$. Then we get a homotopy commutative diagram
    \[
      \xymatrix{
        |L|\ar[r]^{\widetilde{\varphi}_L}\ar[d]&U_L\ar[r]^{u_L}\ar[d]&K(B,2)\ar[d]\\
        |K|\ar[r]^{\widetilde{\varphi}_K}&U_K\ar[r]^{u_K}&K(A,2)
      }
    \]
    where $u_K$ and $u_L$ are isomorphisms in $\pi_2$. Note that $\varphi_L=\varphi_{\dl_K(v)}\sqcup *$. By assumption, $\varphi_{\dl_K(v)}\simeq*$. Then since $\R Z_L^{m-1}$ is path-connected, $\varphi_L$ is null-homotopic, implying $\widetilde{\varphi}_L$ is null-homotopic too. Thus the cohomology class $u_K\circ\widetilde{\varphi}_K$ belongs to the kernel of the map $H^2(K;A)\to H^2(L;A)$. Since this map is identified with $i_v^*\colon H^2(K;A)\to H^2(\dl_K(v);A)$, the cohomology class $u_K\circ\widetilde{\varphi}_K$ belongs to the kernel of the map \eqref{cohomology}.

    By Theorem \ref{cone decomposition}, $\R Z_K^{m-1}$ is a suspension. Then $A\cong H_2(\R Z_K;\Z)\otimes\Z\pi_1(\R Z_K^{m-1})$, and in particular, $A$ is a sum of copies of $H_2(\R Z_K;\Z)$ which is a finitely generated abelian group. Thus by Proposition \ref{dual special}, the map $i_v^*\colon H^2(K;A)\to\bigoplus_{v\in[m]} H^2(\dl_K(v);A)$ is injective, implying $u_K\circ\widetilde{\varphi}_K$ is null-homotopic. Since $\dim K\le 2$, we obtain $\varphi_K$ is null-homotopic.
  \end{proof}

  Now we prove Theorem \ref{main}.

  \begin{proof}[Proof of Theorem \ref{main}]
    The implication (1) $\Rightarrow$ (2) is proved by Proposition \ref{1->2}, and the implication (2) $\Rightarrow$ (3) is proved by Theorem \ref{neighborly} and Proposition \ref{2->3}. The implication (3) $\Rightarrow$ (1) is proved by Theorems \ref{Z_K Tor} and \ref{decomposition}.
  \end{proof}


  \section{Golodness over fields and rings}

  As mentioned in Section 1, it is natural to ask whether or not there is a difference between Golodness over fields and over rings. This section gives an answer by giving a simplicial complex which is Golod over any field but is not Golod over some ring.

  First, we prove that the converse of Proposition \ref{1->2} holds over a field.

  \begin{proposition}
    \label{Golod field}
    Let $\Bbbk$ be a field. A two-dimensional simplicial complex $K$ is Golod over $\Bbbk$ if and only if the following conditions hold:
    \begin{enumerate}
      \item the 1-skeleton of $K$ is chordal;

      \item every full subcomplex of $K$ having vertex-breakable second homology over $\Bbbk$ is 1-neighborly.
    \end{enumerate}
  \end{proposition}

  \begin{proof}
    The proof of Proposition \ref{1->2} implies that if $K$ is Golod over $\Bbbk$, then the two conditions hold.

    Suppose conversely that the two conditions hold. Since $\Bbbk$ is a field, the K\"unneth formula in the strong form holds as $H^*(X\times Y;\Bbbk)\cong H^*(X;\Bbbk)\otimes H^*(Y;\Bbbk)$. It is proved by K\"atthan \cite[Theorem 6.3]{K} that for a simplicial complex $L$ of dimension at most three, the condition for (higher) Massey products in Definition \ref{Golod} is redundant. Then to see that $K$ is Golod over $\Bbbk$, it is sufficient to show that the map $H^*(K_{I_1}*K_{I_2};\Bbbk)\to H^*(K_I;\Bbbk)$ is trivial for $*=1,2$, where $I_1,I_2$ are non-empty subsets of $[m]$ such that $I_1\cap I_2=\emptyset$ and $I_1\cup I_2=I$.

    Let $K^1$ denote the 1-skeleton of $K$. It is proved in \cite[Proposition 3.2]{IK3} that a graph is Golod if and only if it is chordal. Then, in particular, $K^1_I$ is Golod (over $\Bbbk$) for each $\emptyset\ne I\subset[m]$. Consider a commutative diagram
    \[
      \xymatrix{
        H^1(K_{I_1}*K_{I_2};\Bbbk)\ar[d]\ar[r]&H^1(K_I;\Bbbk)\ar[d]\\
        H^1(K^1_{I_1}*K^1_{I_2};\Bbbk)\ar[r]&H^1(K^1_I;\Bbbk)
      }
    \]
    where $I_1,I_2$ are non-empty subsets of $[m]$ such that $I_1\cap I_2=\emptyset$ and $I_1\cup I_2=I$. Since $K^1_I$ is Golod over $\Bbbk$, the lower horizontal arrow is trivial. Then since the vertical arrows are injective, the upper horizontal arrow is trivial too.

    We show that the map $H^2(K_{I_1}*K_{I_2};\Bbbk)\to H^2(K_I;\Bbbk)$ is trivial by induction on $|I|$. For $|I|=2$, the map is trivial because $H^2(K_I;\Bbbk)=0$. Suppose that $K_J$ is Golod for $|J|<|I|$. If $K_I$ has vertex-breakable second homology over $\Bbbk$, then it is 1-neighborly by assumption. Thus by Theorems \ref{Z_K Tor}, \ref{decomposition} and \ref{neighborly}, $K_I$ is Golod over $\Bbbk$. Consider a commutative diagram
    \[
      \xymatrix{
        H^2(K_{I_1}*K_{I_2};\Bbbk)\ar[r]\ar[d]&\bigoplus_{v\in I}H^2(L_1*L_2;\Bbbk)\ar[d]\\
        H^2(K_I;\Bbbk)\ar[r]&\bigoplus_{v\in I}H^2(\dl_{K_I}(v);\Bbbk)
      }
    \]
    where $L_i=K_{I_i}$ for $v\not\in I_i$ and $L_i=\dl_{K_{I_i}}(v)$ for $v\in I_i$. If $K_I$ does not have vertex-breakable second homology over $\Bbbk$, then the lower horizontal arrow is injective. By assumption, the right horizontal arrow is trivial, implying the left horizontal arrow is trivial too. Therefore the proof is complete.
  \end{proof}

  \begin{remark}
    The proof of Proposition \ref{Golod field} does not work over a ring in general because the K\"unneth formula in the strong form does not hold and $\mathrm{Hom}(-,R)$ is not right exact.
  \end{remark}

  We consider the following simplicial complexes $K_1$ and $K_2$, where vertices and edges having the same labels are identified.

  
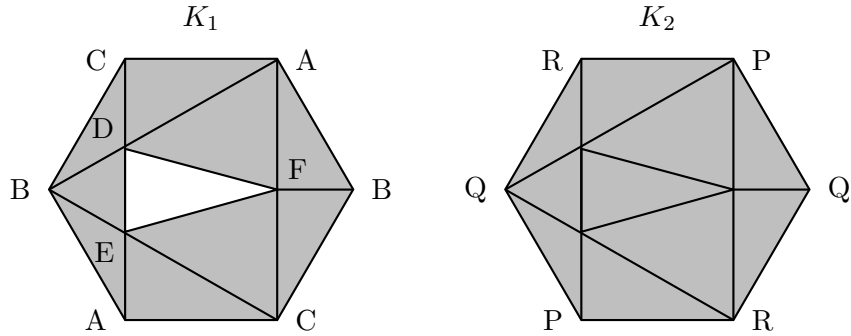
\begin{center}
    \begin{tikzpicture}[x=1cm, y=1cm, thick]
      \filldraw [fill=lightgray] (0,0) -- (0:2) -- ++(60:2) -- ++(120:2) -- ++(180:2) -- ++(240:2) -- cycle ;
      \draw (-1,1.73) -- (2,3.45);
      \draw (-1,1.73) -- (2,0);
      \draw (0,0) -- (0,3.45);
      \draw (2,0) -- (2,3.45);
      \draw (2,1.73) -- (3,1.73);
      \filldraw [fill=white] (2,1.73) -- (0,2.27) -- (0,1.17) -- cycle;
      \node [above] at (1,3.7) {$K_1$};
      \node [left=1mm] at (0,0) {A};
      \node [left=1mm] at (-1,1.73) {B};
      \node [left=1mm] at (0,3.45) {C};
      \node [right=1mm] at (2,3.45) {A};
      \node [right=1mm] at (3,1.73) {B};
      \node [right=1mm] at (2,0) {C};
      \node [above left] at (0,2.27) {D};
      \node [below left] at (0,1.17) {E};
      \node [above right] at (2,1.73) {F};
      \filldraw [fill=lightgray] (6,0) -- (0:8) -- ++(60:2) -- ++(120:2) -- ++(180:2) -- ++(240:2) -- cycle ;
      \draw (5,1.73) -- (8,3.45);
      \draw (5,1.73) -- (8,0);
      \draw (6,0) -- (6,3.45);
      \draw (8,0) -- (8,3.45);
      \draw (8,1.73) -- (9,1.73);
      \draw (8,1.73) -- (6,2.27) -- (6,1.17) -- cycle;
      \node [above] at (7,3.7) {$K_2$};
      \node [left=1mm] at (6,0) {P};
      \node [left=1mm] at (5,1.73) {Q};
      \node [left=1mm] at (6,3.45) {R};
      \node [right=1mm] at (8,3.45) {P};
      \node [right=1mm] at (9,1.73) {Q};
      \node [right=1mm] at (8,0) {R};
    \end{tikzpicture}
  \end{center}\vspace{2mm}

      %

  \noindent Note that
  \[
    |K_1|\simeq S^1\quad\text{and}\quad|K_2|\simeq\R P^2.
  \]
  We define a simplicial complex $M$ by gluing $K_1$ and $K_2$ along the triangles $DEF$ and $PQR$. Since the inclusion of the triangle $DEF$ into $K_1$ is equivalent to the degree two self-map of $S^1$, $M$ is a triangulation of the Moore space $S^1\cup_4e^2$.

  \begin{proof}[Proof of Theorem \ref{M Golod}]
    Let $\Bbbk$ be a field of characteristic 2. By definition, $H_2(M_I;\Bbbk)=0$ unless $M_I$ includes $K_2$, and $M_I$ does not have vertex-breakable second homology over $\Bbbk$ whenever $K_2$ is a proper subcomplex of $M_I$. Thus $M_I$ is Golod over $\Bbbk$ for $M_I\ne K_2$. If $M_I=K_2$, then $M_I$ is 1-neighborly, implying $M_I$ is Golod over $\Bbbk$ by Theorems \ref{Z_K Tor}, \ref{decomposition} and \ref{neighborly}.

    Let $\Bbbk$ be a field of characteristic $\ne 2$. Since $H_2(M_I;\Bbbk)=0$ for each $\emptyset\ne I\subset[m]$, it follows from Proposition \ref{Golod field} that $M$ is Golod over $\Bbbk$. Thus we obtain that $M$ is Golod over any field.

    Since $H_2(\dl_M(v);\Z/4)$ is either 0 or $\Z/2$ for each vertex $v$ and $H_2(M;\Z/4)\cong\Z/4$, $M$ has vertex-breakable second homology over $\Z/4$. Then since $M$ is not 1-neighborly, $M$ is not Golod over some ring by Theorem \ref{main}, where one actually sees from Lemma \ref{vertex-breakable} that $M$ is not Golod over $\Z/4$. Thus the proof is complete
  \end{proof}


\end{document}